\documentclass[12pt]{amsart}

\usepackage{amsmath,amssymb,amsbsy,amsfonts,amsthm,latexsym,
                     amsopn,amstext,amsxtra,euscript,amscd,mathrsfs}
\usepackage{amsmath,amssymb,amsbsy,amsfonts,latexsym,amsopn,amstext,cite,
                                               amsxtra,euscript,amscd,bm}
\usepackage{mathtools}
\usepackage{todonotes}
\usepackage{url}
\usepackage[colorlinks,linkcolor=blue,anchorcolor=blue,citecolor=blue,backref=page]{hyperref}
\usepackage{multirow}
\usepackage{float}

\usepackage[np]{numprint}
\npdecimalsign{\ensuremath{.}}

\def \fT{\mathfrak{T}}

\usepackage{bibentry}

\usepackage[norefs,nocites]{refcheck}

\renewcommand*{\backref}[1]{}
\renewcommand*{\backrefalt}[4]{%
    \ifcase #1 (Not cited.)%
    \or        (p.\,#2)%
    \else      (pp.\,#2)%
    \fi}

\definecolor{olive}{rgb}{0.3, 0.4, .1}
\definecolor{dgreen}{rgb}{0.,0.6,0.}

\newtheorem{theorem}{Theorem}[section]
\newtheorem{lemma}[theorem]{Lemma}

\newtheorem{proposition}[theorem]{Proposition}

\theoremstyle{definition}

\theoremstyle{remark}

\numberwithin{equation}{section}

\newcommand{\li}{\textnormal{li}}

\def\({\left(}
\def\){\right)}
\begin{document}

%\title[Remarks on our paper]{Remarks on our paper 
%%``Multiplicative structure of values of the Euler function"}

%\commI{Alternatively, `Old title, II'}
%%\title[Smooth values of the Euler function]{On smooth values of the Euler function}

\title[Counting integers with a smooth totient]
{Counting integers with a smooth totient}

\author[Banks]{W. D. Banks}
\address{Department of Mathematics, University of Missouri, Columbia, MO 65211, USA}
\email{bankswd@math.missouri.edu}

\author[Friedlander]{J. B. Friedlander}
\address{Department of Mathematics, University of Toronto, Toronto, Ontario M5S 2E4, Canada}
\email{frdlndr@math.utoronto.ca}
 
\author[Pomerance]{C. Pomerance}
\address{Mathematics Department, Dartmouth College, Hanover, NH 03755, USA}
\email{carl.pomerance@dartmouth.edu}

\author[Shparlinski]{I. E. Shparlinski}
\address{Department of Pure Mathematics, University of New South Wales, Sydney, NSW 2052, Australia}
\email{igor.shparlinski@unsw.edu.au}

\date{\today}

\begin{abstract} We fix a gap in our proof of an upper bound for the number of 
positive 
integers $n\le x$ for which the Euler function $\varphi(n)$ has all prime factors at most $y$. 
While doing this we obtain a stronger, likely best-possible result. 
\end{abstract}


\maketitle

\section{Introduction}
Our paper~\cite{BFPS} considers various multiplicative problems related to
Euler's function $\varphi$.  One of these problems concerns the distribution
of integers $n$  for which $\varphi(n)$ is $y$-smooth (or $y$-friable), meaning
that all prime factors of $\varphi(n)$ are at most $y$.  We recall 
that~\cite[Theorem~3.1]{BFPS} asserts that the following bound holds on the quantity
$\Phi(x,y)$ defined to be the number of $n\le x$ such that $\varphi(n)$
is $y$-smooth. 
\begin{quotation}
{\it For any fixed $\varepsilon>0$, numbers $x,y$ with
$y\ge(\log\log x)^{1+\varepsilon}$, and $u=\log x/\log y\to\infty$, we have the bound
$
\Phi(x,y)\le x/\exp((1+o(1))u\log\log u).
$
}
\end{quotation}

Paul Kinlaw has brought to our attention a  flaw in our argument.
 Specifically, in the two-line display near the end of the proof,
our upper bound on the sum $\sum_{p\le y}p^{-c}$ is incorrect for the larger
values of $y$ in our range. 

The purpose of this note is to provide a complete proof of a somewhat stronger
version of~\cite[Theorem~3.1]{BFPS}.  Merging
Propositions~\ref{prop:largeu} and~\ref{prop:smallu}
below we prove the following result.

\begin{theorem}
\label{thm:main}
For any fixed $\varepsilon>0$,  numbers $x,y$
with $y\ge(\log\log x)^{1+\varepsilon}$, and
$u=\log x/\log y\to\infty$, we have
$$
\Phi(x,y)\le x\exp\bigl(-u(\log\log u+\log\log\log u +o(1))\bigr).
$$
\end{theorem}

One might wonder about a matching lower bound for $\Phi(x,y)$,
 but this  is very difficult to
achieve since it depends on the distribution of primes $p$ with $p-1$ being $y$-smooth.
Let $\psi(x,y)$ denote the number of $y$-smooth integers at most $x$, and let
$\psi_\pi(x,y)$ denote the number of primes $p\le x$ such that $p-1$ is $y$-smooth.
It has been conjectured (see~\cite{PS} and the discussion therein) that in a wide range
one has $\psi_\pi(x,y)/\pi(x)\sim\psi(x,y)/x$.  Assuming a weak form of this conjecture,
Lamzouri~\cite{L} has shown 
that there
is a continuous monotonic function $\sigma(u)$ such that
$$
\sigma(u)=\exp\bigl(-u(\log\log u+\log\log\log u+o(1))\bigr)\qquad (u\to\infty),
$$
and such that $\Phi(x,x^{1/u})\sim \sigma(u)x$ as $x\to\infty$ with $u$ bounded.
The function $\sigma$ is explicitly identified as the solution to the integral equation
$$
u\sigma(u)=\int_0^u\sigma(u-t)\rho(t)\,dt,
$$
where $\rho$ is the Dickman--de Bruijn function.  

In light of Lamzouri's theorem, it may be that we have equality in Theorem \ref{thm:main}.

Our proof of Theorem~\ref{thm:main} is given as two results:
Proposition~\ref{prop:largeu} for the case
when $y\le x^{1/\log\log x}$ and  Proposition~\ref{prop:smallu}   
for the case when $y\ge \exp(\sqrt{\log x\log\log x}\,)$.

Note that the ranges of Propositions~\ref{prop:largeu} 
and~\ref{prop:smallu} have a significant overlap. 
In the first range we use a variant of Rankin's trick.  In the second range
we use a variant of the Hildebrand approach~\cite{H} for estimating $\psi(x,y)$.

Our proof is adaptable to multiplicative
functions similar in structure to Euler's $\varphi$-function.
For example, in~\cite{Po} a version of our theorem is used
for the distribution of squarefree $n\le x$ with $\sigma(n)$ being $y$-smooth,
where $\sigma$ is the sum-of-divisors function.
 
In a recent 
paper, Pollack~\cite{P1} shows (as a special case)
that for any fixed number $\alpha>1$,
$$
\Phi(x,(\log x)^{1/\alpha})\le x^{1-(\alpha+o(1))\log\log\log x/\log\log x}
$$
as $x\to\infty$.  A slightly stronger inequality follows from our Theorem~\ref{thm:main}, though in Pollack's result
the inequality applies to sets more general than the $(\log x)^{1/\alpha}$-smooth integers.

Our paper~\cite{BFPS} also considered the distribution of integers $n$
for which $\varphi(n)$ is a square and the distribution of squares
in the image of $\varphi$. These results have attracted interest and
since then have been improved and extended in various ways;
see~\cite{F,FP,P2,PP}.

In what follows, $P(n)$ denotes
the largest prime factor of an integer $n>1$, and $P(1)=1$.
The letter $p$ always denotes a prime number; the letter $n$ always denotes
a positive integer.
As usual in the subject,
we write $\log_kx$ for the $k$th iterate of the natural logarithm,
assuming that the argument is large enough for the expression to make sense.

We use the notations $U = O(V)$  and $U \ll V$ in  their standard
meaning that $|U| \le cV$ for some constant $c$, which throughout this paper may depend
on the real positive parameters $\varepsilon$, $\delta$ $\eta$. 
We also use the notations $U \sim V$  and $U =o(V)$ to indicate that 
$U/V \to 1$ and $U/V \to 0$, respectively, when certain
(explicitly indicated) parameters tend to infinity.

\section{Small $y$}

\subsection{Dickman--de Bruijn function}

 As above, we denote by $\rho$ the Dickman--de Bruijn function;
we refer the reader to~\cite{HildTen} for an exact definition and properties.
For the first range it is useful to have the following two estimates
involving this function. 

\begin{lemma}
\label{lem:useful}
Let $\eta>0$ be arbitrarily small but fixed. 
For $A\ge2$ we have
$$
\sum_{n\ge 1}A^n\rho(n)\ll \exp\left(\frac{(1+\eta)A}{\log A}\right).
$$
\end{lemma}

\begin{proof}
It is sufficient to prove the result for large numbers $A$.
Since $\rho(n)\le 1$, the sum up to $A/(\log A)^2$ is $\ll\exp(A/\log A)$,
hence we need only consider integers $n>A/(\log A)^2$.
We have for $t>1$,
\begin{equation}
\label{eq:rho}
\rho(t)=\exp\(-t\(\log t+\log_2 t -1+\frac{\log_2t-1}{\log t}
+O\(\frac{(\log_2t)^2}{(\log t)^2}\)\)\);
\end{equation}
see for example de Bruijn~\cite[(1.5)]{dB2}.
Consequently, if $n> A/(\log A)^2$ and $A$ is large enough, then
$$
A^n\rho(n)<\exp(n(\log A-\log n-\log_2 n+1)).
$$
In the case $n> A$, this implies that
$$
A^n\rho(n) <\exp(-n\log_2 n +n)<\exp(-n),
$$
and so the contribution to the sum when $n> A$ is $O(1)$.
Now assume that $A/(\log A)^2 < n \le A$.
Let $f(t)=t(\log A-\log t-\log_2 t+1)$.
For any $\theta\ge 1/\log A$ one sees that
\begin{align*}
f\left(\frac{\theta A}{\log A}\right) & =
\frac{\theta A}{\log A} \left(-\log \theta + \log_2 A- \log_2 \left(\frac{\theta A}{\log A}\right) +1 \right)\\
& =-\frac{\theta A}{\log A} (\log \theta+C_{A,\theta}),
\end{align*} 
where  
$$
C_{A,\theta}=\log\left(\frac{\log A+\log\theta-\log_2 A}{\log A}\right)-1.
$$
Hence, when $A$ is large enough depending on the choice of $\eta$, we have
$$
f\left(\frac{\theta A}{\log A}\right)
\le-\frac{\theta A}{\log A} (\log \theta-(1+\eta/2))\qquad(\theta> 1/\log A).
$$
Since this last expression reaches a maximum when $\theta=e^{\eta/2}$, we have
$f(t)\le e^{\eta/2}A/\log A<(1+3\eta/4)A/\log A$ for all $t> A/(\log A)^2$, and so
$$
\sum_{A/(\log A)^2< n\le A}A^n\rho(n)
< A\exp\left(\frac{(1+3\eta/4)A}{\log A}\right)\ll \exp\(\frac{(1+\eta)A}{\log A}\),
$$
which completes the proof of the lemma.
\end{proof}

To prove the main results of this paper, we need information about the distribution of
primes $p$ with $p-1$ suitably smooth.  The following statement, which 
is~\cite[Theorem~1]{PS} (see also~\cite[Equation~(2.3)]{BFPS}), suffices for our purposes.

\begin{lemma}
\label{lem:Smooth Shift}
For $\exp(\sqrt{\log t \log_2 t}\,)\le y\le t$ and with
$u_t=\log t/\log y$ we have
$$
\psi_\pi(t,y)=\sum_{\substack{p\le t\\P(p-1)\le y}}1
\ll u_t\rho\(u_t\)\frac{t}{\log t}=\rho(u_t)\frac{t}{\log y}.
$$
\end{lemma}
 It is useful to observe that the range in Lemma~\ref{lem:Smooth Shift} includes the range 
$$
y \le t\le y^{\log y/2\log_2 y}.
$$

\subsection{Bound on $\Phi(x,y)$ for $(\log_2 x)^{1+\varepsilon}\le y\le x^{1/\log_2 x}$}

We give a proof of the following result.
\begin{proposition}
\label{prop:largeu}
Fix $\varepsilon>0$.
For  $(\log_2 x)^{1+\varepsilon}\le y\le x^{1/\log_2 x}$, and 
$u=\log x/\log y \to\infty$, we have
$$
\Phi(x,y)\le x \exp(-u(\log_2u+\log_3 u+o(1))). % \qquad(u\to\infty).
$$
\end{proposition}

\begin{proof}
We may assume that $u$ is large and shall need to do so at various 
points in the proof.  We may also assume that  $\varepsilon<1$.
Let $\delta>0$ be arbitrarily small but fixed.  We  prove that 
$$\Phi(x,y)\le x\exp(-u(\log_2u+\log_3u-\delta+o(1)))\qquad (u\to\infty),$$
which is sufficient for the desired result. 

Put 
$$
c=1-(\log_2 u+\log_3 u-\delta)/\log y,
$$ 
so that $c<1$ for $u$ sufficiently large.  Also, $u<\log x$ implies that
$$
1-c=\frac{\log_2 u+\log_3u-\delta}{\log y}<\frac{\log_3 x+\log_4 x}{(1+\varepsilon)\log_3 x}<1-\frac\varepsilon 2,
$$
for $u$ sufficiently large, so we may assume that $1>c>\varepsilon/2$.
We have
\begin{equation}
\label{eq:rankin}
\Phi(x,y)\le x^c\sum_{\substack{n\le x\\P(\varphi(n))\le y}}\frac1{n^c}
\le x^c\prod_{\substack{p\le x\\P(p-1)\le y}}\left(1-\frac1{p^c}\right)^{-1}.
\end{equation}
Note that $x^c=x\exp(-u(\log_2 u+\log_3 u-\delta))$, so via~\eqref{eq:rankin} it suffices
to prove that
\begin{equation}
\label{eq:suff}
-\sum_{\substack{p\le x\\P(p-1)\le y}}\log\left(1-\frac1{p^c}\right)=o(u),
\end{equation}
as $u\to\infty$.  This implies that, using $c>\varepsilon/2$,
$$
-\sum_{\substack{p\le x\\P(p-1)\le y}}\log\left(1-\frac1{p^c}\right)=\sum_{\substack{p\le x\\P(p-1)\le y}}\sum_{k\ge1}\frac1{kp^{ck}}
\ll \sum_{\substack{p\le x\\P(p-1)\le y}}\frac1{p^c}.
$$
Hence, to establish~\eqref{eq:suff} and hence 
the desired result, it is sufficient to show that, as $u\to\infty$,
\begin{equation}
\label{eq:suff2}
\sum_{\substack{p\le x\\P(p-1)\le y}}\frac1{p^c}=o(u).
\end{equation}

Put
\begin{equation}
\label{eq:y1y2}
\begin{aligned}
z=\frac{\log y}{2\log_2y},
\end{aligned}
\end{equation}
and consider primes $p\le x$ with $P(p-1)\le y$ in two ranges:
\begin{enumerate}
\item{}
$p\le y^z$, 
\item{} $p>y^z$.
\end{enumerate}
Note that the second range contains primes only in the case that $y^z\le x$.

To estimate the first range for $p$, we have
$$
\sum_{\substack{p\le y^z\\P(p-1)\le y}}\frac1{p^c}
\le\sum_{1\le k< z+1}\,\sum_{\substack{y^{k-1}<p\le y^k \\P(p-1)\le y}}\frac1{p^c}.
$$
For the inner sum we use Lemma~\ref{lem:Smooth Shift} together with partial summation and 
the fact that  
$y^{1-c} = e^{-\delta}\log u\log_2u$ getting that
\begin{align*}
\sum_{\substack{y^{k-1}<p\le y^k \\P(p-1)\le y}}\frac1{p^c} &
\ll\rho(k)\frac{y^{k(1-c)}}{\log y}
+\int_{y^{k-1}}^{y^k}\rho(k-1)\frac1{t^c\log y}\,dt\\
& \ll\rho(k-1)\frac{y^{k(1-c)}}{(1-c)\log y} 
\ll\rho(k-1)\( e^{-\delta}\log u\log_2u\)^k.
\end{align*} 
We use Lemma~\ref{lem:useful} with $A=e^{-\delta}\log u\log_2u$ and 
$\eta=\delta$, finding that
$$
\sum_{\substack{p\le y^z\\P(p-1)\le y}}\frac1{p^c}\ll A\exp\left(\frac{(1+\delta)A}{\log A}\right).
$$  
Since $(1+\delta)A/\log A\sim (1+\delta)e^{-\delta}\log u$ as $u\to\infty$, and
$(1+\delta)e^{-\delta}<1$, this shows that the sum in~\eqref{eq:suff2} is $O(u^{1-\delta'})$
for some $\delta'>0$ depending on the choice of $\delta$.  Thus we have~\eqref{eq:suff2} for primes in the first range.

Now we turn to the second range. As mentioned earlier, we may assume that
$y^z\le x$.  By de Bruijn~\cite[(1.6)]{dB1} we have
\begin{equation}
\label{eq:Brun}
\psi(t,y)\le t/e^{u_t\log u_t}\qquad(y^z<t\le x), 
\end{equation}
where $u_t$ is as in Lemma~\ref{lem:Smooth Shift},
for $u$ sufficiently large.  Ignoring that $p$ is prime we have the bound
\begin{equation}
\label{eq:2nd}
\sum_{\substack{y^z<p\le x\\P(p-1)\le y}}\frac1{p^c}\le\sum_{\substack{y^z-1<n\le x\\P(n)\le y}}\frac1{n^c}\le 
1+ \sum_{z+1\le k \le u}\,\sum_{\substack{y^{k-1}<n\le y^k \\P(n)\le y}}\frac1{n^c}.
\end{equation}

Next, we put
$$
y_0=\exp\left((\log_2 x)^{2}\right)
$$
and consider separately the cases $y\ge y_0$ and $y<y_0$.
In the case that $y\ge y_0$,
using~\eqref{eq:Brun} the inner sum on the right side of~\eqref{eq:2nd} satisfies
\begin{align*}
\sum_{\substack{y^{k-1}<n\le y^k \\P(n)\le y}}\frac1{n^c} &
\le \frac{\psi(y^k,y)}{y^{kc}}+\int_{y^{k-1}}^{y^k}\frac{c\,\psi(t,y)}{t^{c+1}}\,dt\\
&\le k^{-k}y^{k(1-c)}+(k-1)^{-(k-1)}\int_{y^{k-1}}^{y^k}t^{-c}\,dt\\
&\ll k^{-(k-1)}\frac{y^{k(1-c)}}{1-c}\\
&\le k\log y\cdot\exp(-k(\log k-\log_2u-\log_3u+\delta)).
\end{align*}
Since $y\ge y_0$, $k\ge z$,  with  $z$ given by~\eqref{eq:y1y2}, and $u<\log x$, 
we have 
\begin{align*}
\log k-\log_2u-\log_3u &\ge \log z-\log_2u-\log_3u \\
& \ge\log_2 y - \log_3y - \log 2 - \log_2u-\log_3u\\
& \ge \frac{7}{8} \log_2 y   - \log_2u-\log_3u\\
& \ge \frac{7}{4} \log_3 x   - \log_2u-\log_3u >  \frac{1}{2} \log_3 x
\end{align*}
provided that $u$ is large. Hence, 
$$
\sum_{\substack{y^{k-1}<n\le y^k \\P(n)\le y}}\frac1{n^c} \ll \exp(-k)\log y
$$
and so the sum in~\eqref{eq:2nd} is $O\(\exp(-z)\log y\) = O(1)$.

It remains to handle the second range when $y<y_0$.  In this case,
we use an Euler product for a second time, getting that
\begin{align*}
\sum_{\substack{n\le x\\P(n)\le y}}n^{-c}&<\prod_{p\le y}\left(1-p^{-c}\right)^{-1}
\ll\exp\left(\sum_{p\le y}p^{-c}\right)\\
&=\exp\left(\li(y^{1-c})(1+O(1/\log y))+O(|\log(1-c)|)\right),
\end{align*}
where we have used~\cite[Equation~(2.4)]{P} in the last step.
Now
$$
\li(y^{1-c})=(1+o(1))\frac{y^{1-c}}{(1-c)\log y}=\frac{1+o(1)}{e^\delta}\log u,
$$
as $u\to\infty$, and 
$$
|\log(1-c)|<\log_2y<2\log_3x\ll\log_2u.
$$
Therefore
$$
\sum_{\substack{n\le x\\P(n)\le y}}n^{-c}\le u^{e^{-\delta/2}}
$$
for $u$ sufficiently large.  This completes the proof.
\end{proof}

\section{Large $y$}

\subsection{A version of the Hildebrand identity}

We begin this section by proving an analog of the Hildebrand identity which is 
adapted to our function  $\Phi(x,y)$.  Note that it is given as an
inequality, but it would not be hard to account for the excess on the higher side.
\begin{lemma}
\label{lem:ajh}
For $x\ge y\ge2$ we have
$$
\Phi(x,y)\le\frac1{\log x}\int_1^x\frac{\Phi(t,y)}{t}\,dt+\frac1{\log x}\sum_{\substack{d\le x\\P(\varphi(d))\le y}}\Phi\left(\frac xd,y\right)\Lambda(d).
$$
\end{lemma}

\begin{proof}
By partial summation, we have
\begin{equation}
\label{eq:ps}
\sum_{\substack{n\le x\\P(\varphi(n))\le y}}\log n=\Phi(x,y)\log x-\int_1^x\frac{\Phi(t,y)}{t}\,dt.
\end{equation}
On the other hand, we have
\begin{align*}
\sum_{\substack{n\le x\\P(\varphi(n))\le y}}\log n& =
\sum_{\substack{n\le x\\P(\varphi(n))\le y}}\sum_{d\,|\,n}\Lambda(d)
=\sum_{\substack{d\le x\\P(\varphi(d))\le y}}\sum_{\substack{m\le x/d\\P(\varphi(md))\le y}}\Lambda(d)\\
&\le
\sum_{\substack{d\le x\\P(\varphi(d))\le y}}\Phi\left(\frac xd,y\right)\Lambda(d).
\end{align*}
Substituting~\eqref{eq:ps} on the left side and solving the resulting inequality for $\Phi(x,y)$ gives the result.
\end{proof}

\subsection{Bound on $\Phi(x,y)$ for  $y\ge\exp(\sqrt{\log x\log_2 x}\,)$}

\begin{proposition}
\label{prop:smallu}
For %% numbers $x,y$ with 
$y\ge\exp(\sqrt{\log x\log_2 x}\,)$,
and
$u=\log x/\log y\to\infty$, we have
$$
\Phi(x,y)\le x \exp(-u(\log_2u+\log_3 u+o(1))).
$$
\end{proposition}

\begin{proof}
Let $\delta>0$ be arbitrarily small but fixed, and put 
$$
g(u)=\exp(-u(\log_2u+\log_3u-\delta)).
$$
It suffices to show that $\Phi(x,y)\ll xg(u)$
for $x,y$ in the given range. 

For any given $u\ge3$, which without loss of generality we may assume, let $\Gamma_u$ be the supremum of
$\Phi(x,y)/(xg(u))$ for all $x,y$ with $y=x^{1/u}$,
so that trivially $\Gamma_u \le 1/g(u)$. Further, let 
$$
\gamma_u=\sup\{\Gamma_v:3\le v\le u\}.
$$
Our goal is to show that $\gamma_u$ is bounded.  Towards this end, we may assume
that $u\ge u_0\ge3$, where $u_0$ is a suitably large constant, depending on the choice of $\delta$.
 Since $\gamma_u$ is nondecreasing as a function of $u$, we may assume that
\begin{equation}
\label{eq:kingcrimson}
\gamma_u\ge 1\qquad(u\ge u_0), 
\end{equation}
for otherwise $\gamma_u$ is clearly bounded.
We further assume that $u_0$ is large enough so that
\begin{equation}
\label{eq:jarre}
\frac1{\log v}+\frac1{\log v\log_2v}\le\delta\qquad(v\ge u_0). 
\end{equation}

Let $N$ be such that
$$
u_0\le N\le \exp(\sqrt{\log x/\log_2 x}\,)-1.
$$
We claim that for $u_0$ large enough
\begin{equation}
\label{eq:Claim}
\sup_{N<u\le N+1} \Gamma_u \le \gamma_N.
\end{equation}
 By induction, this implies that
$\gamma_u\le\gamma_{u_0}$ for all $u\ge u_0$, and therefore
$$
\Phi(x,y)\le\gamma_{u_0}xg(u)
$$ 
for all $u\ge u_0$, and the result would follow. 

One other observation is that $g(u)\sim e^{-\delta} g(u+1)\log u\log_2u$
as $u\to\infty$, so that with $u_0$ large and $u_0\le N<u\le N+1$, we have
\begin{equation}
\label{eq:gugn}
g(N)\le g(u)\log u\log_2u~\hbox{ and }~g(N-1)\le g(u)(\log u\log_2u)^2.
\end{equation}

To establish~\eqref{eq:Claim} we first consider the term
$$
\fT_1=\frac1{\log x}\int_1^x\frac{\Phi(t,y)}t\,dt
$$
in Lemma~\ref{lem:ajh}.  We split the range of integration as follows:
$$
\int_1^x=\int_1^{y^{u_0}}+\int_{y^{u_0}}^{y^N}+\int_{y^N}^x.
$$

We have trivially that
\begin{equation}
\label{eq:term1-1}
\int_1^{y^{u_0}}\frac{\Phi(t,y)}t\,dt< y^{u_0}.
\end{equation}
We show that for $u_0$ sufficiently large, we have 
\begin{equation}
\label{eq:term1-1+}
y^{u_0}\le xg(u)/g(u_0).
\end{equation}
Since $y^{u_0}=x^{u_0/u}$, \eqref{eq:term1-1+} is equivalent to showing that for 
$$
D(u)=\(1-\frac{u_0}u\)\log x-\log g(u_0)- u(\log_2u+\log_3u-\delta)\ge0,
$$
we have
\begin{equation}
\label{eq:term1-1++}
D(u)\ge0.
\end{equation}
Note that the hypothesis $y\ge\exp(\sqrt{\log x\log_2x})$ implies 
that $\log x >u^2(\log_2u+\log_3u)$.
By considering $D'(u)$
and using \eqref{eq:jarre}, we see that $D(u)$ is increasing for $u\ge u_0$
and $u_0$ sufficiently large.  Since $D(u_0)=0$, this proves~\eqref{eq:term1-1++}, which
establishes~\eqref{eq:term1-1+}, and so via~\eqref{eq:term1-1} we have
\begin{equation}
\label{eq:term1-1final}
\int_1^{y^{u_0}}\frac{\Phi(t,y)}t\,dt\le xg(u)/g(u_0).
\end{equation}

Also,
$$
\int_{y^{u_0}}^{y^N}\frac{\Phi(t,y)}t\,dt \le \gamma_N I,
$$
where 
$$
I= \int_{y^{u_0}}^{y^N}g(\log t/\log y)\,dt
= \int_{u_0}^N g(v)y^v\log y\,dv 
= \int_{u_0}^N g(v)\,d(y^v) .
$$  
Thus, $I$ is equal to
\begin{align*}
y^vg(v)\Big|_{u_0}^N &
+\int_{u_0}^N\left(\log_2 v+\log_3v
-\delta+\frac1{\log v}+\frac1{\log v\log_2v}\right)g(v)y^v\,dv\\
&\qquad\qquad\qquad \qquad \qquad \qquad \quad <y^Ng(N)+\frac{\log_2N+\log_3N}{\log y}I,
\end{align*}
where we have used~\eqref{eq:jarre} in the last step.
Assuming $u_0$ is sufficiently large (and thus so are $x$  and $y$), we see that
\begin{equation}
\label{eq:term1-2}
\int_{y^{u_0}}^{y^N}\frac{\Phi(t,y)}t\,dt<2\gamma_Ny^Ng(N). 
\end{equation}

Finally,
\begin{equation}
\label{eq:term1-3}
\int_{y^N}^x\frac{\Phi(t,y)}t\,dt \le\int_{y^N}^x\frac{\Phi(t,t^{1/N})}t\,dt
\le \gamma_N g(N)(x-y^N).
\end{equation}

Thus, using~\eqref{eq:term1-1final}, \eqref{eq:term1-2},  and  \eqref{eq:term1-3},
 we have 
\begin{equation}
\begin{split}
\label{eq:term1}
\fT_1
& \le\frac{xg(u)}{g(u_0)\log x}+\frac {2\gamma_Nx}{\log x}g(N)\\
& \le \frac{2\gamma_N \log u\log_2u+1/g(u_0)}{\log x}xg(u),
\end{split}
\end{equation}
assuming that $u_0$ is sufficiently large, where we used \eqref{eq:gugn} for the last step.

Next, we consider the second term
$$
\fT_2 =\frac1{\log x}\sum_{\substack{d\le x\\P(\varphi(d))\le y}}\Phi\left(\frac xd,y\right)\Lambda(d)
$$
in Lemma~\ref{lem:ajh}, and begin
by estimating the contribution from terms $d\le y$.  For such $d$ we have
$y^{N-1} \le x/d$  (since $y^N\le x$), which implies that $y\le(x/d)^{1/(N-1)}$.
Hence, this part of $\fT_2$
is at most
\begin{align*}
\frac1{\log x}\sum_{d\le  y}\Phi\left(\frac x{d},y\right)\Lambda(d)
&\le\frac1{\log x}\sum_{d\le  y}
\Phi\left(\frac x{d},\(\frac xd\)^{1/(N-1)}\right)\Lambda(d)\\
&\le \frac{\gamma_Nx}{\log x}g(N-1)\sum_{d\le y}\frac{\Lambda(d)}{d}.
\end{align*}  
Hence, by the Mertens formula
\begin{equation}
\begin{split}
\label{eq:term2-1}
\frac1{\log x}\sum_{d\le  y}\Phi\left(\frac x{d},y\right)\Lambda(d)
&\le \frac{2\gamma_Nx\log y}{\log x}g(N-1) \\
&\le\frac{2\gamma_Nx(\log u\log_2u)^2}{u}g(u),
\end{split}
\end{equation}
assuming that $u_0$ is sufficiently large and using \eqref{eq:gugn}.

Next, we consider  the contribution from terms 
$d=p^a>y$ for which $p\le y$ (and thus the positive integer $a$ is at least two), 
finding from the trivial bound  $\Phi\(x/p^a,y\) \le x/p^a$ that
\begin{equation}
\label{eq:term2-1.5}
\frac1{\log x}\sum_{\substack{p\le  y\\p^a>y}}\Phi\left(\frac x{p^a},y\right)\log p
\le\frac{x}{\log x}\sum_{\substack{p\le y\\p^a>y}}\frac{\log p}{p^a}
\ll\frac{x}{\sqrt{y}\log x}.
\end{equation}

The remaining terms are of the form $d=p^a$ with $p>y$, and since
$P(\varphi(d))\le y$ we conclude that $a=1$, i.e., $d=p$.
Therefore, we need to estimate
\begin{equation}
\label{eq:term2-2a}
\frac1{\log x}\sum_{\substack{y<p\le x\\P(p-1)\le y}}\Phi\left(\frac xp,y\right)\log p\\
=\frac1{\log x}\sum_{1\le k<u}S_k,
\end{equation}
where
$$
S_k=\sum_{\substack{y^k<p\le \min\{x,y^{k+1}\}\\P(p-1)\le y}}
\Phi\left(\frac xp,y\right)\log p.
$$
We also denote
$$
T_k=\sum_{\substack{y^k<p\le \min\{x,y^{k+1}\}\\P(p-1)\le y}}\frac{\log p}p.
$$
For integers $k\le u/2$ we use the bound
$$
S_k\le \gamma_N x g(u-k-1)T_k
\le \gamma_N x\log u\log_2u\cdot g(u-k)T_k,
$$
whereas for larger integers $k>u/2$, the trivial bound
$\Phi\(x/p,y\)\le x/p$ and~\eqref{eq:kingcrimson} together imply that
$$
S_k\le \gamma_N x T_k;
$$
consequently, using \eqref{eq:term2-2a},
\begin{equation}
\begin{split}
\label{eq:term2-2b}
&\frac1{\log x}\sum_{\substack{y<p\le x\\P(p-1)\le y}}\Phi\left(\frac xp,y\right) \log p\\
&\qquad\le\frac{\gamma_N x\log u\log_2u}{\log x}\sum_{1\le k\le u/2} g(u-k)T_k
+\frac{\gamma_N x}{\log x}\sum_{u/2<k<u}T_k.
\end{split}
\end{equation}

Next, define 
$$
h(k)=\exp(-k(\log k+\log_2(k{+}1)-1))
$$
and note that from~\eqref{eq:rho} we have
\begin{equation}
\label{eq:rho vs h}
k\rho(k)\ll h(k).
\end{equation}
By partial summation, using Lemma~\ref{lem:Smooth Shift}  together 
with~\eqref{eq:rho vs h}, we see that 
 there is an absolute constant $c_0$ such that for $1\le k<u$ we have
$$
T_k=\sum_{\substack{y^k< p\le \min\{x,y^{k+1}\}\\P(p-1)\le y}}\frac{\log p}p
 \le c_0h(k)\log y.
$$
Using this bound in~\eqref{eq:term2-2b} along with the simple bound
$$
h(k)\le \frac{g(u)}{u}\qquad(k>u/2)
$$
leads to
\begin{equation}
\begin{split}
\label{eq:term2-2c}
&\frac1{\log x}\sum_{\substack{y<p\le x\\P(p-1)\le y}}\Phi\left(\frac xp,y\right) \log p\\
&\qquad\le\frac{c_0\gamma_N x\log u\log_2u}{u}\sum_{1\le k\le u/2} g(u-k)h(k)
+\frac{c_0\gamma_N x}{u}g(u). 
\end{split}
\end{equation}

To bound the sum in~\eqref{eq:term2-2c}, we start with the estimate
\begin{equation}
\label{eq:tacos}
\log g(u-k)=-(u-k)(\log_2u+\log_3u-\delta)+O\(\frac{k}{\log u}\),
\end{equation}
which holds uniformly for $1\le k\le u/2$.  Using~\eqref{eq:tacos}
and assuming that $u_0$ is sufficiently large depending on $\delta$,
we derive that
\begin{equation}
\label{eq:bound-tiful1}
g(u-k)h(k)\le g(u) e^{B_u(k)}\qquad(1\le k\le u/2),
\end{equation}
where
$$
B_u(k)=k(\log_2u+\log_3u-\log k-\log_2(k{+}1)+1-\delta/2).
$$
Note that
\begin{align*}
\frac{d B_u(k)}{d k} & = \log_2u+\log_3u-\delta/2 \\
& \qquad \qquad - \log k-\log_2(k+1) - \frac{k}{(k+1)\log(k+1)}\\&  = \log\(e^{-\delta/2}\frac{\log u }{k}\)+ \log \frac{\log_2u}{\log (k+1)} - \frac{k}{(k+1)\log(k+1)}.
\end{align*}
Therefore, the function $B_u$ reaches its maximum for some $k =k_0$ with
$$
k_0=e^{-\delta/2}\log u+O\(\frac{\log u}{\log_2u}\)
$$
and, since for a constant $C> 0$ 
the derivative is bounded independently of $u$  for any $k$ in the interval
$$
k \in \left [e^{-\delta/2}\log u - C\frac{\log u}{\log_2u}, e^{-\delta/2}\log u+ C\frac{\log u}{\log_2u}\right],
$$ 
we obtain 
\begin{align*}
\max_{1\le k\le u/2}B_u(k)&=B_u\(e^{-\delta/2}\log u\)+O\(\frac{\log u}{\log_2u}\)\\
&=
e^{-\delta/2}\log u+O\(\frac{\log u}{\log_2u}\).
\end{align*}
This implies via~\eqref{eq:bound-tiful1} that
\begin{equation}
\label{eq:bound-tiful2}
\max_{1\le k\le u/2}g(u-k)h(k)\le g(u) u^{1-\delta/3}\qquad(1\le k\le u/2),
\end{equation}
if $u_0$ is sufficiently large.  Moreover, for any fixed constant $c>1$,
it is easy to see that $B_u$ is decreasing for $k\ge c\log u$ if $u_0$ is
sufficiently large depending on $\delta$ and $c$, and after a simple estimate
we have 
$$
\max_{c\log u\le k\le u/2}B_u(k)\le (c-c\log c)\log u.
$$
In particular, with $c=3$ (and noting that $3 -3 \log 3 = -0.295\cdots$), this implies via~\eqref{eq:bound-tiful1} that
\begin{equation}
\label{eq:bound-tiful3}
\max_{3\log u\le k\le u/2}g(u-k)h(k)\le g(u) u^{-1/4}.
\end{equation}
Splitting the range of the summation in~\eqref{eq:term2-2c}
according to whether $k\le 3\log u$ or $k>3\log u$, and
using~\eqref{eq:bound-tiful2} and~\eqref{eq:bound-tiful3}, respectively,
we have
\begin{align*}
\sum_{1\le k\le u/2} g(u-k)h(k)
&\le\sum_{1\le k\le 3\log u} g(u) u^{1-\delta/3}+\sum_{3\log u<k\le u/2} g(u) u^{-1/4}\\
&\le3g(u) u^{1-\delta/3}\log u+g(u) u^{3/4}\\
&\le g(u)u^{1-\delta/4}
\end{align*}
if $\delta$ is small enough and
$u_0$ sufficiently large.  Inserting this bound into~\eqref{eq:term2-2c},
it follows that
\begin{equation}
\begin{split}
\label{eq:term2-2}
\frac1{\log x}&\sum_{\substack{y<p\le x\\P(p-1)\le y}}\Phi\left(\frac xp,y\right) \log p\\
&\qquad\le\frac{c_0\gamma_N x\log u\log_2u}{u}g(u)u^{1-\delta/4}
+\frac{c_0\gamma_N x}{u}g(u)\\
&\qquad\le c_0\gamma_Nu^{-\delta/5}g(u)x,
\end{split}
\end{equation}
again assuming that $u_0$ is large.

Combining the bounds~\eqref{eq:term2-1}, \eqref{eq:term2-1.5} and~\eqref{eq:term2-2}
we obtain
\begin{equation}
\begin{split}
\label{eq:term2}
\fT_2 \le \frac{2\gamma_Nx(\log u\log_2u)^2}{u}g(u) & + c_0\gamma_Nu^{-\delta/5}g(u)x \\
& \qquad + O\(\frac{x}{\sqrt{y}\log x}\).
\end{split}
\end{equation}

We deduce from Lemma~\ref{lem:ajh} and the bounds~\eqref{eq:term1}  and~\eqref{eq:term2}, that for $u_0$ large,
$$
\Phi(x,y)\le\gamma_Ng(u)x.
$$
This establishes our claim \eqref{eq:Claim}, and the proposition is proved.
\end{proof}

\section{Comments} 
The bound of Proposition~\ref{prop:largeu}, taken at the lower range with
$y = (\log_2 x)^{1+\varepsilon}$, and thus with
$$
u = \frac{\log x}{(1+\varepsilon)\log_3 x}\,,
$$
 implies that 
\begin{align*}
\Phi\(x, (\log_2 x)^{1+\varepsilon}\)&\le  x \exp\(-\frac{\log x}{1+\varepsilon} +O\(\frac{\log x \log_4x}{\log_3 x}\)\)\\
& = x^{\varepsilon/(1+\varepsilon) +o(1)},
\end {align*}
Hence $\Phi\(x, \log_2 x\) = x^{o(1)}$. 
Although we do not have any lower bounds for this range that are much
better than the trivial bound $\Phi(x,y)\ge \psi(x,y)$, this does suggest the existence
of a phase transition near the point $y = \log_2 x$.  Using the same heuristic
as in Erd\H os \cite{E}, one should have quite small values of $y$ with $\Phi(x,y)=x^{1-o(1)}$.
In particular this should hold for any $y$ of the shape $(\log x)^\varepsilon$, with $\varepsilon>0$ fixed.
 It is interesting to recall that 
for the classical function $\psi(x,y)$ there is a well-known phase transition near the point $y = \log x$;
see~\cite{dB2}. 

\section*{Acknowledgement}
 
The authors are very grateful to Paul Kinlaw  for  pointing out a problem in our 
previous paper~\cite{BFPS}. 

 This work  was partially supported by NSERC (Canada) Discovery Grant A5123 (for J.F.) and the  Australian Research Council 
 Grant DP170100786 (for I.S.).

\bibliographystyle{amsplain}

\begin{thebibliography}{10}

\bibitem{BFPS} W. D. Banks, J. B. Friedlander, C. Pomerance, and I. E. Shparlinski,
{\it Multiplicative structure of values of the Euler function}, in High Primes and Misdemeanours:
Lectures in honour of the sixtieth birthday of Hugh Cowie Williams,
A. van der Poorten, ed., Fields Inst.\ Comm. {\bf41} (2004), pp. 29--47.

\bibitem{dB1} N. G.~de~Bruijn, {\it On the number of positive integers $\leq x$ and free of prime factors $>y$}, Nederl.\ Acad.\ Wetensch., Proc.\ \ Ser.\ A \textbf{54} (1951), 50--60.

\bibitem{dB2} \bysame, {\it On the number of positive integers $\leq x$ and free of prime factors $>y$, II}, Indag.\ Math.\ \textbf{28} (1966), 239--247.

\bibitem{E} P. Erd\H os, {\it On the normal number of prime factors of $p - 1$ and some related problems concerning
Euler's $\varphi$-function}, Q. J. Math., Oxford Ser.\ {\bf6} (1935), 205--213.

\bibitem{F} T. Freiberg, {\it Products of shifted primes simultaneously taking perfect power values},
J. Aust.\ Math.\ Soc.\ {\bf92} (2012) ,145--154.

\bibitem{FP} T. Freiberg and C. Pomerance, {\it A note on square totients},
Int.\ J. Number Theory {\bf11} (2015), 2265--2276. 

\bibitem{H} A. J. Hildebrand, {\it On the number of positive integers $\le x$ and 
free of prime factors $>y$}, J. Number Theory {\bf 22} (1986), 289--307.

\bibitem{HildTen} A. J. Hildebrand and G. Tenenbaum, {\it Integers without large prime factors\/},   J. Th\'eor. Nombres Bordeaux {\bf 5} (1993), 411--484.


\bibitem{L} Y. Lamzouri, {\it Smooth values of the iterates of 
the Euler $\varphi$-function}, Canadian J. Math.\  {\bf 59} (2007), 127--147.

\bibitem{P1} P. Pollack, {\it Popular subsets for Euler's $\varphi$-function}, Preprint, 2018.

\bibitem{P2} \bysame, {\it How often is Euler's totient a perfect power?}, Preprint, 2018.

\bibitem{PP} P. Pollack and C. Pomerance, {\it Square values of Euler's function},
 Bull.\ London Math.\ Soc.\ {\bf46} (2014), 403--414.
 
\bibitem{P} C. Pomerance, {\it Two methods in elementary analytic number theory}, in
Number theory and applications, R. A. Mollin, ed.,  Kluwer Academic Publishers, Dordrecht, 1989, pp. 135--161.

\bibitem{Po}
\bysame, {\it On amicable numbers}, in Analytic number theory (in honor of
Helmut Maier's 60th birthday), M. Rassias and C. Pomerance, eds.,
Springer, Cham, Switzerland, 2015, pp.\ 321--327.

\bibitem{PS} C. Pomerance and I. E. Shparlinski, {\it Smooth orders and cryptographic applications},
Proc. ANTS-V, Sydney, Australia, Springer Lecture Notes in Computer Science {\bf 2369}, (2002), pp. 338--348.

\end{thebibliography}

\end{document}